\documentstyle[twoside, amsmath, amssymb, ]{article}

\newtheorem{theorem}{Theorem}[section]
\newtheorem{lemma}[theorem]{Lemma}
\newtheorem{remark}[theorem]{Remark}
\newtheorem{proposition}[theorem]{Proposition}
\newtheorem{corollary}[theorem]{Corollary}
\newtheorem{definition}[theorem]{Definition}

\newenvironment{proof}{\begin{trivlist} \item[]{\em Proof.}}{\end{trivlist}}
\newcount\refno
\newcommand\be{\begin{equation}}
\newcommand\ee{\end{equation}}
\newcommand\bn{\begin{eqnarray}}
\newcommand\en{\end{eqnarray}}
\newcommand\bns{\begin{eqnarray*}}
\newcommand\ens{\end{eqnarray*}}
\newcommand\bd{\begin{definition}}
\newcommand\ed{\end{definition}}
\newcommand\br{\begin{remark}}
\newcommand\er{\end{remark}}
\newcommand\bt{\begin{theorem}}
\newcommand\et{\end{theorem}}
\newcommand\bp{\begin{proposition}}
\newcommand\ep{\end{proposition}}
\newcommand\bc{\begin{corollary}}
\newcommand\ec{\end{corollary}}
\newcommand\bl{\begin{lemma}}
\newcommand\el{\end{lemma}}
\newcommand\pf{\begin{proof}}
\newcommand\qed{\end{proof}\eop}

\newcommand\bR{{\mathbb R}}
\newcommand\bN{{\mathbb N}}

\newcommand\cR{{\cal R}}

\def\eop{\hfill\rule{2.0mm}{2.0mm}}

\makeindex

\begin{document}

\title{Duals of Bernoulli Numbers and Polynomials and Euler Number and Polynomials}
\author{Tian-Xiao He and Jinze Zheng \\
{\small Department of Mathematics}\\
{\small Illinois Wesleyan University}\\
{\small Bloomington, IL 61702-2900, USA}\\
}
\date{}
\maketitle

\begin{abstract}
\noindent A sequence inverse relationship can be defined by a pair of infinite inverse matrices. If the pair of matrices are the same, they define a dual relationship. Here presented is a unified approach to construct dual relationships via pseudo-involution of Riordan arrays. Then we give four dual relationships for Bernoulli numbers and Euler numbers, from which the corresponding dual sequences of Bernoulli polynomials and Euler polynomials are constructed. Some applications in the construction of identities of Bernoulli numbers and polynomials and Euler numbers and polynomials are discussed based on the dual relationships. 

\vskip .2in \noindent AMS Subject Classification: 05A15, 05A05, 11B39,
11B73, 15B36, 15A06, 05A19, 11B83.

\vskip .2in \noindent {\bf Key Words and Phrases:} inverse matrices, dual, Bernoulli numbers, Bernoulli polynomials, Euler numbers, Euler polynomials, Riordan arrays, pseudo-involution. 
\end{abstract}



\setcounter{page}{1} \pagestyle{myheadings} 
\markboth{T. X. He and J. Zheng
}
{Duals of Bernoulli and Euler Numbers and Polynomials}


\section{Introduction}

Krattenthaler defines a class of matrix inverses in \cite{Kra}, which has numerous famous special cases presented by Gould and Hsu \cite{GH73}, Carlitz \cite{Car}, Bressoud \cite{Bre}, Chu and Hsu \cite{CH}, Ma\cite{Ma}, etc. Hsu, Shiue, and one of the authors study Riordan matrix inverses in \cite{HHS}. Let infinite low triangle matrices $(d_{n,k})_{0\leq k\leq n}$ and $(\bar d_{n,k})_{0\leq k\leq n}$ be matrix inverses. 
Then a  sequence inverse relationship can be defined as 

\be\label{-1-0}
f_n=\sum^n_{k=0}d_{n,k}g_k \Longleftrightarrow g_n=\sum^n_{k=0}\bar d_{n,k} f_k,
\ee
where sequences $\{f_n\}_{n\geq 0}$ and $\{g_n\}_{n\geq 0}$ are called the inverse sequences with respect to inverse matrices $(d_{n,k})_{0\leq k\leq n}$ and $(\bar d_{n,k})_{0\leq k\leq n}$. If $d_{n,k}=\bar d_{n,k}$, i.e., $\{ d_{n,k}\}_{0\leq k\leq n}$ is self-inverse, then $\{f_n\}_{n\geq 0}$ and $\{g_n\}_{n\geq 0}$ are said to be a pair of dual sequences (or they are dual each other) with respect to self-inverse matrix $(d_{n,k})_{0\leq k\leq n}$, i.e., they satisfy 

\be\label{-1-1}
f_n=\sum^n_{k=0}d_{n,k}g_k \Longleftrightarrow g_n=\sum^n_{k=0} d_{n,k} f_k,
\ee
If there exists a sequence $\{f_n\}_{n\geq 0}$ that is dual to itself with respect to a self-inverse matrix $(d_{n,k})_{0\leq k\leq n}$, i.e., 

\be\label{-1-2}
f_n=\sum^n_{k=0}d_{n,k}f_k,
\ee
then, $\{ f_n\}_{n\geq 0}$ is called self-dual sequence with respect to the matrix $(d_{n,k})_{0\leq k\leq n}$.

Let $d_{n,k}=\binom{n}{k} (-1)^k$. Then $(d_{n,k})_{0\leq k\leq n}$ is a self-inverse matrix because 

\[
\sum^n_{k=0}\sum^k_{j=0} \binom{n}{k}\binom{k}{j}(-1)^{k+j}=\delta_{n,j},
\]
where $\delta_{n,j}$ is the Kronecker symbol.  Let $\{ a_n\}_{n\geq 0}$ be a sequence of complex numbers. Then $\{ a^\ast_n\}_{n\geq 0}$ defined by 

\be\label{0-0}
a^\ast_n=\sum^n_{k=0} \binom{n} k (-1)^ka_k
\ee
is the dual sequence of $\{ a_n\}_{n\geq 0}$ (see, for example, Graham, Knuth, and Patashnik \cite{GKP}) with respect to $(\binom{n}{k} (-1)^k)$. Hence, 

\be\label{0-1}
a_n=\sum^n_{k=0} \binom{n} k (-1)^ka^\ast_k. 
\ee
$\{ a_n\}_{n\geq 0}$ and $\{ a^\ast_n\}_{n\geq 0}$ are a pair of inverse sequences with $(a^\ast)^\ast_n=a_n$. 
If $a^\ast_n=a_n$, then $\{ a_n\}_{n\geq 0}$ is called the self-dual sequence (see Z. W. Sun \cite{Sun01}). 
For instance, the following number sequences are self-dual sequences with respect to the dual relationship \eqref{0-0} (see, for example, Z. H. Sun \cite{SunZH}): 

\[
\left\{ \frac{1}{2^{n}}\right\},\,\, \left\{ \frac{1}{\binom{n+2m-1}{m}}\right\},\,\, \{ (-1)^{n}B_{n}\},\,\, \{ L_{n}\},\,\, \{ nF_{n-1}\},
\]
where $\{ B_{n}\}$, $\{ L_{n}\}$, and $\{ F_{n}\}$ are the Bernoulli sequence, Lucas sequence, and Fibonacci sequence.

Bernoulli polynomials $B_n(x)$ and Euler polynomials $E_n(x)$ for $n=0,1,\ldots$ are defined by 

\be\label{0-2}
\frac{te^{xt}}{e^t-1}=\sum_{n\geq 0}B_n(x)\frac{t^n}{n!}\,\, \text{and}\,\, 
\frac{2e^{xt}}{e^t+1}=\sum_{n\geq 0}E_n(x)\frac{t^n}{n!}.
\ee
Bernoulli numbers $B_n$ and Euler numbers $E_n$ for $n=0,1,\ldots$ are defined by 

\be\label{0-3}
B_n:=B_n(0)
\quad \text{and}\quad E_n=2^n E_n\left( \frac{1}{2}\right).
\ee
A large literature scatters widely in books and journals on Bernoulli numbers $B_{n}$, and Bernoulli polynomials $B_{n}(x)$ and  Euler numbers $E_n$ and Euler polynomials $E_n(x)$. They can be studied by means of the binomial expressions connecting them, 

\bn\label{0-4}
&&B_{n}(x)=\sum^{n}_{k=0}\binom{n}{k}B_{k}x^{n-k}, \quad n\geq 0,\\
&&E_n(x)=\sum^n_{k=0} \binom{n}{k}\left( x-\frac{1}{2}\right)^{n-k}\frac{E_k}{2^k}, \quad n\geq 0,\label{0-5}
\en
where $E_k=2^kE_k(1/2)$. The study brings consistent attention of researchers working in combinatorics, number theory, etc. 

This paper will study the duals of Bernoulli number sequence and Euler number sequence with respect to $(\binom{n}{k} (-1)^k)$ and other self-inverse matrices. The dual sequence of  $\{ B_n\}_{n\geq 0}$ with respect to $(\binom{n}{k} (-1)^k)$ is denoted by $\{ B^\ast_n\}_{n\geq 0}$, and the corresponding dual Bernoulli polynomials denoted by $\{ B^{\ast}_{n}(x)\}_{n\geq 0}$ is defined similarly to \eqref{0-4} by using the the duals of sequences $\{ B_n\}_{n\geq 0}$, i.e.,

\bn\label{0-6}
&&B^\ast_{n}(x)=\sum^{n}_{k=0}\binom{n}{k}B^\ast_{k}x^{n-k}, \quad n\geq 0,
\en
where $B^{\ast}_{n}$ are the duals of the Bernoulli numbers $\{ B_n\}_{n\geq 0}$, i.e.,  

\bn\label{0-8}
&&B^\ast_n=\sum^n_{k=0}\binom{n}{k} (-1)^k B_k, \quad n\geq 0.
\en

Hence, there are several questions raised: (1) With respect to which self-inverse matrix the Bernoulli number sequence $\{B_n\}$ is self-dual? (2) What is the relationship between the self-inverse matrix in (1) and the self-inverse matrix for $\{ (-1)^nB_n\}_{n\geq 0}$, $(\binom{n}{k} (-1)^k)$? And (3) does there exist a unified approach to construct self-inverse matrices and self-duals with respect to a self-inverse matrix? We will answer those questions by using the Riordan array theory.

Riordan arrays are infinite, lower triangular matrices defined by the
generating function of their columns. They form a group, called the
Riordan group (see Shapiro et al. \cite{SGWW}). Some of the main results on
the Riordan group and its application to combinatorial sums and identities
can be found in \cite{CJ}--\cite{CKS3}, \cite{DFR05}--\cite{GH}, 
\cite{He08}--\cite{JLN}, \cite{LMMS}-\cite{LMMS2}, \cite{MRSV}-\cite{MSV}, \cite{Nkw}, 
\cite{PW}-\cite{PW00}, \cite{Rio}-\cite{Sha2}, \cite{Spr1}--\cite{Spr2}, and \cite{Yan}-\cite{WW}.

More formally, let us consider the set of formal power series (f.p.s.) $%
\mbox{$\mathcal{F}$} = {\Bbb R}[\![$$t$$]\!]$; the order of $f(t) \in %
\mbox{$\mathcal{F}$}$, $f(t) =\sum_{k=0}^\infty f_kt^k$ ($f_k\in {\Bbb R}$),
is the minimal number $r\in\mbox{$\mathbb{N}$}$ such that $f_r \neq 0$; $%
\mbox{$\mathcal{F}$}_r$ is the set of formal power series of order $r$. It
is known that $\mbox{$\mathcal{F}$}_0$ is the set of invertible f.p.s.
and $\mbox{$\mathcal{F}$}_1$ is the set of compositionally invertible 
f.p.s., that is, the f.p.s. $f(t) $ for which the compositional inverse $%
\bar f(t) $ exists such that $f(\bar f(t) ) =\bar f(f(t) ) = t$. Let $d(t)
\in \mbox{$\mathcal{F}$}_0$ and $h(t) \in \mbox{$\mathcal{F}$}_1$; the pair $%
(d(t) ,\,h(t) )$ defines the (proper) Riordan array $%
D=(d_{n,k})_{n,k\in \mbox{\scriptsize${\mathbb{N}}$}}=(d(t), h(t))$ having 
\begin{equation}  \label{Radef}
d_{n,k} = [t^n]d(t) h(t) ^k
\end{equation}
or, in other words, having $d(t) h(t)^k$ as the generating function whose
coefficients make-up the entries of column $k$.

It is immediately to be known that the usual row-by-column product of two Riordan
arrays is also a Riordan array: 
\begin{equation}  \label{Proddef}
(d_1(t) ,\,h_1(t) ) \ast (d_2(t) ,\,h_2(t) ) = (d_1(t) d_2(h_1(t)
),\,h_2(h_1(t) )),
\end{equation}
where the fundamental theorem of Riordan arrays

\begin{equation}  \label{FTR}
(d(t),\, h(t)) f(t)=d(t)f(h(t))
\end{equation}
is applied to obtain \eqref{Proddef}. The Riordan array $I = (1,\,t)$ is everywhere 0 except that it
contains all $1$'s on the main diagonal; it can be easily proved that $I$
acts as an identity for this product, that is, $(1,\,t) \ast (d(t) ,\,h(t) )
= (d(t) ,\,h(t) ) \ast (1,t) = (d(t) ,\,h(t) )$. From these facts, we deduce
a formula for the inverse Riordan array: 
\begin{equation}  \label{Invdef}
(d(t) ,\,h(t) )^{-1} = \left( \frac{1}{d({\bar h}(t) )},\,{\bar h }(t)
\right)
\end{equation}
where ${\bar h }(t) $ is the compositional inverse of $h(t) $. In this way,
the set ${\cal R}$ of proper Riordan arrays forms a group.

From \cite{Rog}, an infinite lower triangular array $[d_{n,k}]_{n,k\in{\bN}}=(d(t), h(t))$ is a Riordan array if and only if a sequence $A=(a_0\not= 0, a_1, a_2,\ldots)$ exists such that for every $n,k\in{\bN}$ there holds 
\be\label{eq:1.1}
d_{n+1,k+1} =a_0 d_{n,k}+a_1d_{n,k+1}+\cdots +a_nd_{n,n},
\ee 
which is shown in \cite{HS} to be equivalent to 
\be\label{eq:1.2}
h(t)=tA(h(t)).
\ee
Here, $A(t)$ is the generating function of the $A$-sequence. In \cite{HS, MRSV} it is also shown that a unique sequence $Z=(z_0, z_1,z_2,\ldots)$ exists such that every element in column $0$ can be expressed as the linear combination 
\be\label{eq:1.3}
d_{n+1,0}=z_0 d_{n,0}+z_1d_{n,1}+\cdots +z_n d_{n,n},
\ee
or equivalently,
\be\label{eq:1.4}
d(t)=\frac{d_{0,0}}{1-tZ(h(t))}.
\ee
We call sequences $A$ and $Z$ the $A$-(characterization) sequence and $Z-$(characterization) sequence of the Riordan array $(d(t),h(t))$, respectively, and the generating functions of called $A$-sequence and $Z-$sequence the $A$ (generating) function and $Z$ (generating) function, respectively.

In next section, we present a unified approach to construct a class of self-inverse matrices and their application for the construction of dual number sequences and dual polynomial sequences. In Section $3$ we will discuss the algebraic structure of dual number sequences with respect to the dual relationships established in Section $2$. Some identities of self-dual number sequences are found accordingly. In Section $4$, more identities of dual number sequences will be constructed by using the dual relationships.

\section{Construction of self-inverse matrices and the corresponding self-duals}
It is calear that a Riordan array $(d(t), h(t))$ and its inverse 

\[
(d(t), h(t))^{-1}=\left( \frac{1}{d(\bar h(t))}, \bar h(t)\right), \quad \text{where}\quad \bar h(h(t))=h(\bar h(t))=t,
\]
is a pair of inverse matrices. If $(d(t), h(t))^{-1}=(d(t), h(t))$, i.e., $(d(t), h(t))$ is an involution, i.e., it has an order of $2$, 
then $(d(t), h(t))$ is a self-inverse matrix. Here, if $g$ is an element of a group $G$, then the smallest positive integer 
$n$ such that $g^n = e$, the identity of the group, if it exists, is called the order of $g$. If there is no such
integer, then $g$ is said to have infinite order. It is well-known that (see Shapiro \cite{Sha1}) if
we restrict all entries of a Riordan array to be integers, then any element of finite order in the
Riordan group must have order 1 or 2, and each element of order 2 generates a subgroup of order 2.
Sprugnoli and one of the author find the characterization of Riordan arrays of order 2 in \cite{HS}.

In combinatorial situations, a Riordan array often has nonnegative integer entries and hence it can not have order $2$. Therefore, we consider (see definition, for example, in \cite{CK}) an element $R\in{\cR}$ to have pseudo-order $2$, i.e., $RM$ has order $2$, where $M=(1,-t)$. Those $R$ are called pseudo-Riordan involutions or briefly pseudo-involutions (see Cameron and Nkwanta \cite{CN} and \cite{Sha1}). We now present some sufficient and necessary conditions to identify pseudo-involutions. 
 
 \begin{theorem}\label{thm:-1-0}
 Let $(d(t), h(t))$ be a pseudo-involution, and let $A(t)$ and $Z(t)$ be the generating functions of the $A$-sequence and $Z$-sequence of $(d(t), h(t))$, i.e., $A$-function and $Z-$function. Then the following statements are equivalent to the statement that the Riordan array $(d(t), h(t))$ is a pseudo-involution: 
 
 (1) $\pm (d(t), h(t))(1,-t)=(\pm d(t), -h(t))$ are involutions.
 
 (2) $(1,-t) (d(t), h(t))(1,-t)=(d(t), h(t))^{-1}$, the inverse of $(d(t), h(t))$.
 
 (3) $\pm (1,-t)(d(t), h(t))=(\pm d(-t), h(-t))$ are involutions.
 
 (4) $A(t)=\frac{-t}{h(-t)}$ and $Z(t)=\frac{d(-t)-1}{h(-t)}$. 
\end{theorem} 

 \begin{proof}
 Let $(d(t), h(t))$ be a pseudo-involution. Then $(d(t), h(t))(1,-t)=(d(t), -h(t))$ is an involution, i.e., 
 
 \[
 (d(t), -h(t))(d(t), -h(t))=(d(t) d(-h(t)), -h(-h(t))=(1,t).
 \]
 Additionally, $-(d(t), h(t))(1,-t)=(-d(t), -h(t))$ satisfies 
 
 \[
 (-d(t), -h(t))(-d(t), -h(t))=(d(t), -h(t))(d(t), -h(t))=(1,t),
 \]
 which implies $-(d(t), h(t))(1,-t)=(-d(t), -h(t))$ is also an involution. Hence, we have proved the assumption that $(d(t), h(t))$ is a pseudo-involution implies (1). 
 (2) follows from (1) because  
 
 \[
I= \left((d(t), h(t))(1,-t)\right)\left( (d(t), h(t))(1,-t)\right) =(d(t), h(t))\left( (1,-t)(d(t), h(t))(1,-t)\right).
\]
From the above equations, we have 

\bns
&&\left(\pm (1,-t)(d(t), h(t))\right) \left(\pm (1,-t)(d(t), h(t))\right)\\
&=&\left((1,-t)(d(t), h(t))(1,-t)\right) (d(t), h(t))=(1,t)=I.
\ens
Hence, we obtain (3) from (2). Similarly, (1) follows from (3). 

To find the $A$-function and $Z$-function of pseudo-involution $(d(t), h(t))$, we recall (see Theorem 3.3 of \cite{HS}) the $A$-sequence of 

\[
(d_{3}(t), h_{3}(t))=(d_{1}(t), h_{1}(t))(d_{2}(t), h_{2}(t))
\]
has the $A$-(generating) function

\[
A_{3}(t)=A_{2}(t)A_{1}\left( \frac{t}{A_{2}(t)}\right),
\]
where $A_{1}(t)$ and $A_{2}(t)$ are the $A$-(generating) functions of $(d_{1}(t), h_{1}(t))$ and $(d_{2}(t), h_{2}(t))$, respectively. 
Since the $A$-functions of $(d(t), h(t))$ and $(1,-t)$ are respectively $A(t)$ and $-1$, the $A$-function of $(d(t), h(t))(1,-t)$ is 

\be\label{-1-3}
A_{3}(t)=-A\left(-t\right).
\ee
On the other hand, from Theorem 4.3 of \cite{HS} we know the $A$-(generating) function of the involution $(d(t), h(t))(1,-t)=(d(t), -h(t))$ is 

\be\label{-1-4}
A_{3}(t)=\frac{t}{-h(t)}.
\ee
Comparing \eqref{-1-3} and \eqref{-1-4} we have 

\[
A(t)=\frac{-t}{h(-t)}.
\]
From Theorem 3.4 of \cite{HS}, the $Z$-(generating) function of 

\[
(d_{3}(t), h_{3}(t))=(d(t), h(t))(1,-t)=(d(t), -h(t))
\]
is 

\[
Z_{3}(t)=Z(-t)
\]
because the $Z$-functions of $(d(t), h(t))$ and $(1,-t)$ are $Z(t)$ and $0$, respectively. Since $(d_{3}(t), h_{3}(t))$ is an involution, from Theorem 4.3 of \cite{HS} we also have 

\[
Z_{3}(t)=\frac{1-d(t)}{-h(t)}.
\]
Comparing the last two equations, we immediately know 

\[
Z(t)=\frac{d(-t)-1}{h(-t)}.
\]
From the definition of $A$-sequence and $Z$-sequence, a Riordan array $(d(t), h(t))$ that possesses the above $A$-function and $Z$-function is a pseudo-involution. 
 \end{proof}
 \eop

\begin{corollary}\label{cor:-1-0}
Let nonzero $d(t)\in \mbox{$\mathcal{F}$}_0$ and $h(t) \in \mbox{$\mathcal{F}$}_1$. Then the infinite lower triangle matrix $(d(t), h(t))$ is a pseudo-involution if and only if 

\[
\bar h(t)=-h(-t)\quad \text{and} \quad d(t)=\frac{h(-t)}{h(-t)-td(-t)+t},
\]
where the denominator is assumed not being zero. 
\end{corollary}

\begin{proof}
$\bar h(t)=-h(-t)$ if and only if 

\[
\frac{t}{\bar h(t)}=\frac{-t}{h(-t)},
\]
or equivalently, there exists the function $A(t)$ such that 

\[
A(t)=\frac{t}{\bar h(t)} \quad \text{and}\quad A(t)=\frac{-t}{h(-t)},
\]
which implies that $(d(t), h(t))$ has an $A$-function (see, for example, \cite{HS}) satisfying (4) of Theorem \ref{thm:-1-0}. 

$d(t)=h(-t)/(h(-t)-td(-t)+t)$ if and only if 

\[
\frac{d(-t)-1}{h(-t)}=\frac{d(t)-1}{td(t)},
\]
or equivalently, there exists the function $Z(t)$ such that 

\[
Z(t)=\frac{d(t)-1}{td(t)} \quad \text{and}\quad Z(t)=\frac{d(-t)-1}{h(-t)},
\]
which implies that $(d(t), h(t))$ has a $Z$-function (see, for example, \cite{HS}) which satisfies (4) of Theorem \ref{thm:-1-0}.  Combining all above statements together we 
know $(d(t), h(t))$ is a pseudo-involution, completing the proof of Corollary \ref{cor:-1-0}.
\end{proof}
\eop

Corollary \ref{cor:-1-0} also presents an algorithm to find a pseudo-involution. Generally one can accomplish this by carrying through the procedure demonstrated by the following example as suggested by the corollary. For instance, it is clear that $h(t)=t/(1-t)$ ( or $t/(1+t)$) has the compositional inverse  $\bar h(t)=t/(1+t)$ (or $t/(1-t)$) and satisfies $\bar h(t)=-h(-t)$. From Corollary \ref{cor:-1-0}, we may also find that $d(t)=1/(1-t)$ (or $1/(1+t)$) satisfies $d(t)=h(-t)/(h(-t)-td(-t)+t)$. Therefore $(1/(1-t), t/(1-t))$ ($(1/(1+t), t/(1+t))$) is a pseudo-involution. 

\noindent{\bf Remark} It can be seen that $(d(t), h(t))$ is a pseudo-involution if and only if there exist $A$-function and $Z$-function satisfying four relationships 

\[
A(t)=\frac{t}{\bar h(t)},\,\,  A(t)=\frac{-t}{h(-t)}, \,\, Z(t)=\frac{d(t)-1}{td(t)}, \,\, \text{and}\,\, Z(t)=\frac{d(-t)-1}{h(-t)},
\]
which not only give the formulas shown in Corollary \ref{cor:-1-1} but also the following useful formula in the construction of pseudo-involutions:

\[
h(t)=\frac{tZ(t)}{Z(-t)(1-tZ(t))}\quad \text{and}\quad d(t)=\frac{h(t)Z(-t)}{tZ(t)}.
\]
The proof is straightforward from the above four relationships and is omitted.

From (1) and (3) of Theorem \ref{thm:-1-0} we obtain the following result. 

\begin{corollary}\label{cor:-1-1}
$(1/(1-t), t/(1-t))=\left( \binom{n}{k}\right)_{n,k\geq 0}$ and $(1/(1+t), t/(1+t))=\left( \binom{n}{k}(-1)^{n-k}\right)_{n,k\geq 0}$ are pseudo-involutions. Hence, from (1) and (3) of Theorem \ref{thm:-1-0}, we generate the following four Riordan involutions, denoted by $R_{1}$, $R_{2}$, $R_{3}$, and $R_{4}$, respectively.  

\bn\label{-1-5}
R_{1}&=&\left( \frac{1}{1-t}, \frac{t}{1-t}\right)(1,-t)=(1,-t)\left( \frac{1}{1+t}, \frac{t}{1+t}\right)\nonumber\\
&=&\left( \frac{1}{1-t}, \frac{-t}{1-t}\right)=\left( \binom{n}{k}(-1)^{k}\right)_{n,k\geq 0},\nonumber\\
R_{2}&=&-\left( \frac{1}{1-t}, \frac{t}{1-t}\right)(1,-t)=-(1,-t)\left( \frac{1}{1+t}, \frac{t}{1+t}\right)\nonumber\\
&=&\left( \frac{1}{t-1}, \frac{t}{t-1}\right)=\left( \binom{n}{k}(-1)^{k+1}\right)_{n,k\geq 0},\nonumber\\
R_{3}&=&(1,-t)\left( \frac{1}{1-t}, \frac{t}{1-t}\right)=\left( \frac{1}{1+t}, \frac{t}{1+t}\right)(1,-t)\nonumber\\
&=&\left( \frac{1}{1+t}, \frac{-t}{1+t}\right)=\left( \binom{n}{k}(-1)^{n}\right)_{n,k\geq 0},\nonumber\\
R_{4}&=&-(1,-t)\left( \frac{1}{1-t}, \frac{t}{1-t}\right)=-\left( \frac{1}{1+t}, \frac{t}{1+t}\right)(1,-t)\nonumber\\
&=&\left( -\frac{1}{1+t}, \frac{-t}{1+t}\right)=\left( \binom{n}{k}(-1)^{n+1}\right)_{n,k\geq 0}.
\en
\end{corollary}

\begin{theorem}\label{thm:-1-2}
Let $R_{i}$ ($i=1,2,3,4$) be the Riordan arrays shown in Corollary \ref{cor:-1-1}. Then there hold the following four dual relationships, denoted by $D_{1}$, $D_{2}$, $D_{3}$, and $D_{4}$, respectively.

\bns
&& D_{1}: a_{n}=\sum^{n}_{k=0}\binom{n}{k}(-1)^{k}a_{k},\\
&& D_{2}: a_{n}=\sum^{n}_{k=0}\binom{n}{k}(-1)^{k+1}a_{k},\\
&& D_{3}: a_{n}=\sum^{n}_{k=0}\binom{n}{k}(-1)^{n}a_{k},\\
&& D_{4}: a_{n}=\sum^{n}_{k=0}\binom{n}{k}(-1)^{n+1}a_{k}.\\
\ens
Furthermore, $\{ (-1)^{n}B_{n}\}_{n\geq 0}$ and $\{ B_{n}\}_{n\geq 0}$ are self-dual sequences with respect to $D_{1}$ and $D_{3}$, respectively. $\left\{ E_{n}\left( \frac{1}{2}\right)-\frac{1}{2^{n}}\right\}$ and $\left\{ (-1)^{n}\left( E_{n}\left( \frac{1}{2}\right)-\frac{1}{2^{n}}\right)\right\}$ are self-dual sequences with respect to $D_{2}$ and $D_{4}$, respectively.
\end{theorem}

\begin{proof}
Since (see, for example, Apostol \cite{Apo} and  Milton and Stegun \cite{MS}) 

\be
\sum^{n}_{k=0}\binom{n}{k} B_{k}=B_{n}(1)=(-1)^{n}B_{n}, 
\ee
$\{ (-1)^{n}B_{n}\}_{n\geq 0}$ and $\{ B_{n}\}_{n\geq 0}$ are self-dual sequences with respect to $D_{1}$ and $D_{3}$, respectively. Similarly, from \eqref{0-5} and $(-1)^{n}E_{n}(-x)=-E_{n}(x)+2x^{n}$ (see, for example \cite{Apo, MS}) there hold  

\bns
&&\sum^{n}_{k=0}\binom{n}{k}(-1)^{k+1}\left( E_{k}\left( \frac{1}{2}\right)-\frac{1}{2^{k}}\right)\\
&=&-(-1)^{n}\sum^{n}_{k=0}\binom{n}{k}\left(- \frac{1}{2}-\frac{1}{2}\right)^{n-k} E_{k}\left( \frac{1}{2}\right)+\sum^{n}_{k=0}\binom{n}{k}(-1)^{k}\frac{1}{2^{k}}\\
&=&-(-1)^{n}E_{n}\left( -\frac{1}{2}\right) +\frac{1}{2^{n}}\\
&=&E_{n}\left( \frac{1}{2}\right)-2\frac{1}{2^{n}}+\frac{1}{2^{n}}=E_{n}\left(\frac{1}{2}\right)-\frac{1}{2^{n}}
\ens
Similarly, we have 

\bns
&&\sum^{n}_{k=0}\binom{n}{k}(-1)^{n+1}(-1)^{k}\left( E_{k}\left( \frac{1}{2}\right)-\frac{1}{2^{k}}\right)\\
&=&-\sum^{n}_{k=0}\binom{n}{k}(-1)^{n-k}E_{k}\left( \frac{1}{2}\right)+(-1)^{n}\sum^{n}_{k=0}\binom{n}{k}(-1)^{k}\frac{1}{2^{k}}\\
&=&-E_{n}\left( -\frac{1}{2}\right)+(-1)^{n}\frac{1}{2^{n}}\\
&=&(-1)^{n}\left( E_{n}\left( \frac{1}{2}\right)-\frac{1}{2^{n}}\right),
\ens
completing the proof of the theorem.  
\end{proof}
\eop

We now consider the duals of Bernoulli and Euler numbers and the corresponding duals of Bernoulli and Euler polynomials with respect to different dual relationships shown in Theorem \ref{thm:-1-2}. 

\begin{theorem}\label{thm:1.1}
Let $B^\ast_n$ be the duals of Bernoulli numbers $B_{n}$ with respect to $R_{1}$, and let $B^\ast_n(x)$ be the corresponding dual Bernoulli polynomials defined by \eqref{0-6}. 
Then there hold

\be\label{0-10}
B^\ast_n(x)=(-1)^nB_n(-x-1)
\ee
for all $n\geq 0$, 
and 

\be\label{0-12}
B^\ast_n=(-1)^n B_n+n.
\ee
\end{theorem}

\begin{proof}
\eqref{0-4} and \eqref{0-6} give 

\bns
&&B^{\ast}_{n}(x)=\sum^{n}_{k=0}\binom{n}{k}(x)^{n-k}B^{\ast}_{k}\\
&=&\sum^{n}_{k=0}\binom{n}{k}(x)^{n-k}\sum^{k}_{j=0}\binom{k}{j}(-1)^{j}B_{j}\\
&=&\sum^{n}_{j=0}\binom{n}{j}(-1)^{j}B_{j}\sum^{n}_{k=j}\binom{n-j}{k-j}(x)^{n-k}\\
&=&\sum^{n}_{j=0}\binom{n}{j}(-1)^{j}B_{j}(x+1)^{n-j}\\
&=&(-1)^{n}B_{n}(-x-1).
\ens
Hence, 

\bns
&&B^{\ast}_{n}=B^{\ast}_{n}(0)=(-1)^{n}B_{n}(-1)\\
&=&B_{n}(1)+n=(-1)^{n}B_{n}(0)+n,
\ens
which implies \eqref{0-12}.
\end{proof}
\eop

\begin{corollary}\label{cor:1.2}
Let $B^\ast_n(x)$ be the duals of Bernoulli polynomials. Then their generating function is 

\bn\label{0-13}
&&\sum_{n\geq 0} B^\ast_n (x)\frac{t^n}{n!}=\sum_{n\geq 0} (-1)B_n(-x-1)\frac{t^n}{n!}\nonumber \\
&=&\frac{-t e^{(x+1)t}}{e^{-t}-1}=\frac{e^{(x+1)t}}{1+\frac{1-t-e^{-t}}{t}},
\en
\end{corollary}

Other duals of Bernoulli and Euler numbers can be presented below.

\begin{theorem}\label{thm:1.2}
With respect to the dual relationship $D_{3}$, the duals of numbers $(-1)^{n}B_{n}$, denoted by $\left((-1)^{n}B_{n}\right)^{\ast}$, may be written as  

\be\label{0-14}
\left((-1)^{n}B_{n}\right)^{\ast}=\sum^{n}_{k=0}\binom{n}{k}(-1)^{n}(-1)^{k}B_{k}=B_{n}+(-1)^{n}n.
\ee
And the corresponding dual Bernoulli polynomials are 

\be\label{0-15}
\sum^{n}_{k=0}\binom{n}{k} \left((-1)^{k}B_{k}\right)^{\ast} x^{n-k}=B_{n}(x) -n(x-1)^{n-1}.
\ee

With respect to  the dual relationship $D_{4}$, the duals of numbers $E_{n}(1/2)-(1/2)^{n}$, denoted by $\left(E_{n}(1/2)-(1/2)^{n}\right)^{\ast}$, have the 
expressions 

\bn\label{0-16}
&&\left(E_{n}\left( \frac{1}{2}\right)-\left( \frac{1}{2}\right)^{n}\right)^{\ast}= \sum^{n}_{k=0}\binom{n}{k}(-1)^{n+1}\left(E_{k}\left( \frac{1}{2}\right)-\left( \frac{1}{2}\right)^{k}\right)
\nonumber\\
&=& (-1)^{n}\left(E_{n}\left( \frac{1}{2}\right)+\frac{3^{n}-2}{2^{n}}\right).
\en
The corresponding dual Euler polynomials are 

\be\label{0-17}
\sum^{n}_{k=0}\binom{n}{k} \left( x-\frac{1}{2}\right)^{n-k}\left(E_{k}\left( \frac{1}{2}\right)-\left( \frac{1}{2}\right)^{k}\right)^{\ast}=(-1)^{n}E_{n}(-x+1)-2(x-1)^{n}+(x-2)^{n}.
\ee

With respect to the dual relationship $D_{2}$, the duals of numbers $(-1)^{n}\left(E_{n}(1/2)-(1/2)^{n}\right)$, denoted by $\left((-1)^{n}\left(E_{n}(1/2)-(1/2)^{n}\right)\right)^{\ast}$, can  be evaluated as  

\bn\label{0-18}
&&\left((-1)^{n}\left(E_{n}\left( \frac{1}{2}\right) -\frac{1}{2^{n}}\right)\right)^{\ast}=\sum^{n}_{k=0}\binom{n}{k}(-1)^{k+1}(-1)^{k}\left(E_{k}\left( \frac{1}{2}\right) -\frac{1}{2^{k}}\right)
\nonumber\\
&=&E_{n}\left( \frac{1}{2}\right)+\frac{3^{n}-2}{2^{n}}.
\en
And the corresponding dual Bernoulli polynomials are 

\be\label{0-19}
\sum^{n}_{k=0}\binom{n}{k} \left( x-\frac{1}{2}\right)^{n-k}\left((-1)^{k}\left(E_{k}\left( \frac{1}{2}\right) -\frac{1}{2^{k}}\right)\right)^{\ast}=E_{n}(x) +(x+1)^{n}-2x^{n}.
\ee
\end{theorem}

\begin{proof}
The proofs are straightforward from the definitions and are omitted. 
\end{proof}
\eop

\section{Generating functions of self-dual number sequences}

In this section, we give some structures of self-dual number sequences, by using which the characterizations,  
relationships, and other properties of the self-dual number sequences can be obtained. The first half of the following theorem is 
presented in Prodinger \cite{Pro}. 

\begin{theorem}\label{thm:1.3}
Let $a(x)=\sum_{n\geq 0}a_{n}x^{n}$. Then its coefficient sequence is a self-dual sequence with respect to $D_{1}$ ($D_{2}$),  
i.e., it satisfies 

\[
\sum^{n}_{k=0}\binom{n}{k}(-1)^{k}a_{k}=a_{n}\quad (-a_{n})
\]
for $n\in {\bN}\cup\{0\}$ if and only if $a(x)$ satisfies the equation 

\[
a\left( \frac{x}{x-1}\right)=(1-x) a(x) \quad (-(1-x)a(x)), 
\]
respectively. The coefficient sequence of $a(x)$ is a self-dual sequence with respect to $D_{3}$ ($D_{4}$), i.e., it satisfies 

\[
\sum^{n}_{k=0}\binom{n}{k}(-1)^{n}a_{k}=a_{n}\quad (-a_{n})
\]
for $n\in {\bN}\cup\{0\}$ if and only if $a(x)$ satisfies the equation 

\[
a\left( -\frac{x}{1+x}\right)=(1+x) a(x) \quad (-(1+x)a(x)),
\]
respectively. 
\end{theorem}

\begin{proof}
Since the first half of the theorem is given in \cite{Pro}. It is sufficient to prove the second half. 

\bns
&&\pm \frac{1}{1+x}a\left( -\frac{x}{1+x}\right)=\pm \sum_{k=0}^{\infty} a_{k}(-1)^{k}x^{k}(1+x)^{-k-1}\\
&=&\pm \sum_{k=0}^{\infty} a_{k}(-1)^{k}x^{k}\sum^{\infty}_{j=0}\binom{-k-1}{j}x^{j}\\
&=&\pm \sum_{k=0}^{\infty} \sum^{\infty}_{j=0}a_{k}(-1)^{k+j}\binom{k+j}{j}x^{k+j}\\
&=&\pm \sum_{n=0}^{\infty} \left((-1)^{n}\sum^{n}_{j=0}\binom{n}{j}a_{n-j}\right)x^{n}=\pm\sum_{n=0}^{\infty}a_{n}x^{n}=\pm a(x).
\ens
The proof is complete. 
\end{proof}
\eop

\begin{corollary}\label{cor:1.4}
Let $\{ a_{n}\}_{n\geq 0}$ be a given sequence with ordinary generating function $a(x)=\sum_{n\geq 0}a_{n}x^{n}$. Then 

(1) \cite{SunZH} $\{ a_{n}\}_{n\geq 0}$ is a self-dual sequence with respect to $D_{1}$ if and only if $\{ 2a_{n+1}-a_{n}\}_{n\geq 0}$ is a self-dual sequence with respect to $D_{2}$.

(2) $\{ a_{n}\}_{n\geq 0}$ is a self-dual sequence with respect to $D_{3}$ if and only if $\{ 2a_{n+1}+a_{n}\}_{n\geq 0}$ is a self-dual sequence with respect to $D_{4}$.
\end{corollary}

\begin{proof} 
The result shown in (1) is given in \cite{SunZH}. It is sufficient to prove (2). Let 

\[
b_{n}=2a_{n+1}+a_{n}
\]
with its ordinary generating function $b(x)=\sum_{n\geq 0}b_{n}x^{n}$. Thus,

\[
b(x)=\frac{2(a(x)-a_{0})}{x}+a(x)=\frac{x+2}{x}a(x)-\frac{2}{x}a_{0},
\]
which implies 

\bns
&&b\left( -\frac{x}{1+x}\right)=-\frac{x+2}{x}a\left( -\frac{x}{1+x}\right) +\frac{2}{x}(1+x)a_{0}\\
&=&-(1+x)\left( \frac{x+2}{x}\frac{1}{1+x}a\left( -\frac{x}{1+x}\right) -\frac{2}{x}a_{0}\right).
\ens
Therefore, 

\[
a\left( -\frac{x}{1+x}\right)=(1+x) a(x)
\]
if and only if 

\bns
&&b\left( -\frac{x}{1+x}\right)=-(1+x)\left( \frac{x+2}{x}a(x) -\frac{2}{x}a_{0}\right)=-(1+x)b(x).
\ens
Based on Theorem \ref{thm:1.3} we have finished the proof. 
\end{proof}
\eop

\begin{theorem}\label{thm:1.5}
Let $\{ a_{n}\}_{n\geq 0}$ be a given sequence with exponential generating function $a^{\ast}(x)=\sum_{n\geq 0}a_{n}x^{n}/n!$. Then

(1) \cite{SunZH} $\{ a_{n}\}_{n\geq 0}$ is a self-dual sequence with respect to $D_{1}$ if and only if $a^\ast (x)e^{-x/2}$ is an even function.

(2) \cite{SunZH} $\{ a_{n}\}_{n\geq 0}$ is a self-dual sequence with respect to $D_{2}$ if and only if $a^\ast (x)e^{-x/2}$ is an odd function.

(3) $\{ a_{n}\}_{n\geq 0}$ is a self-dual sequence with respect to $D_{3}$ if and only if $a^\ast (x)e^{x/2}$ is an even function.

(4) $\{ a_{n}\}_{n\geq 0}$ is a self-dual sequence with respect to $D_{4}$ if and only if $a^\ast (x)e^{x/2}$ is an odd function.
\end{theorem}

\begin{proof}
The results shown in (1) and (2) are given in \cite{SunZH}. We only need to prove (3) and (4) here by using a similar argument:

\bns
&&a^\ast (-x)e^{-x}=\sum_{k\geq 0}(-1)^{k}a_{k}\frac{x^{k}}{k!}\sum_{j\geq 0}(-1)^{j}\frac{x^{j}}{j!}\\
&=&\sum_{k\geq 0}\sum_{j\geq 0}(-1)^{k+j}a_{k}\frac{x^{k+j}}{k!j!}\\
&=&\sum_{n\geq 0}\left(\sum^{n}_{k=0}(-1)^{n}\binom{n}{k}a_{k}\right) \frac{x^{n}}{n!}.
\ens
Hence, if 

\[
\sum^{n}_{k=0}(-1)^{n}\binom{n}{k}a_{k}= a_{n}\quad \text{and}\quad -a_{n},
\]
then   

\[
a^{\ast}(-x)e^{-x}= a(x)\quad \text{and} \quad -a(x),
\]
respectively, which can be written briefly as  

\[
a^{\ast}(-x)e^{-x/2}=\pm a(x)e^{x/2}.
\]
If the case of positive sign on the right-hand side holds, i.e., $\{a_{n}\}$ is a self-dual sequence with respect to $D_{3}$, then $a^{\ast}(x)e^{x/2}$ is an even function; while the negative sign holds, or equivalently, $\{a_{n}\}$ is a self-dual sequence with respect to $D_{4}$, then the function $a^{\ast}(x)e^{x/2}$ is odd. It is easy to see the sufficiencies of (3) and (4) are also true.  This concludes the proof of the theorem.
\end{proof}
\eop

\cite{SunZH} uses Theorem \ref{thm:1.5} to derive a numerous identities. \cite{Wang} uses umbral calculus to reprove and extend some of them. We now survey their results and extend them to other self-dual sequences. 

\begin{theorem}\label{thm:1.6}
For any function $f$, we have 

(1) \cite{SunZH} $\sum^{n}_{k=0}\binom{n}{k} \left( f(k)-(-1)^{n-k}\sum^{k}_{j=0}\binom{k}{j}f(j)\right) a_{n-k}=0$ for $n\in {\bN}\cup\{0\}$ if $\{ a_{n}\}_{n\geq 0}$ is a self-dual sequence with respect to $D_{1}$.

(2) \cite{Wang} $\sum^{n}_{k=0}\binom{n}{k} \left( f(k)+(-1)^{n-k}\sum^{k}_{j=0}\binom{k}{j}f(j)\right) a_{n-k}=0$ for $n\in {\bN}\cup\{0\}$ if $\{ a_{n}\}_{n\geq 0}$ is a self-dual sequence with respect to $D_{2}$.

(3) $\sum^{n}_{k=0}\binom{n}{k} \left( f(k)-\sum^{k}_{j=0}(-1)^{n-j}\binom{k}{j}f(j)\right) a_{n-k}=0$ for $n\in {\bN}\cup\{0\}$ if $\{ a_{n}\}_{n\geq 0}$ is a self-dual sequence with respect to $D_{3}$.

(4) $\sum^{n}_{k=0}\binom{n}{k} \left( f(k)+\sum^{k}_{j=0}(-1)^{n-j}\binom{k}{j}f(j)\right) a_{n-k}=0$ for $n\in {\bN}\cup\{0\}$ if $\{ a_{n}\}_{n\geq 0}$ is a self-dual sequence with respect to $D_{4}$.
\end{theorem}

\begin{proof} The proofs of (3) and (4) are similar as (1) and (2) by using either Theorem \ref{thm:1.5} or umbral calculus. Hence, we omit  them. 
\end{proof}
\eop

From Theorem \ref{thm:-1-2}, $\{ (-1)^{n}B_{n}\}_{n\geq 0}$ and $\{ B_{n}\}_{n\geq 0}$ are self-dual sequences with respect to $D_{1}$ and $D_{3}$, respectively. $\left\{ E_{n}\left( \frac{1}{2}\right)-\frac{1}{2^{n}}\right\}$ and $\left\{ (-1)^{n}\left( E_{n}\left( \frac{1}{2}\right)-\frac{1}{2^{n}}\right)\right\}$ are self-dual sequences with respect to $D_{2}$ and $D_{4}$, respectively. Therefore, we may use Theorem \ref{thm:1.6} to obtain the following identities.

\begin{theorem}\label{thm:1.7}
Let $B_n$ and $E_n$ be Bernoulli numbers and Euler numbers, respectively. For any function $f$, we have 

\bn\label{0-20}
&&\sum^{n}_{k=0}\binom{n}{k} \left( (-1)^{n-k}f(k)-\sum^{k}_{j=0}\binom{k}{j}f(j)\right) B_{n-k}=0,\nonumber\\
&&\sum^{n}_{k=0}\binom{n}{k} \left( f(k)+(-1)^{n-k}\sum^{k}_{j=0}\binom{k}{j}f(j)\right) \left( E_{n-k}\left( \frac{1}{2}\right)-\frac{1}{2^{n-k}}\right)=0,\nonumber\\
&&\sum^{n}_{k=0}\binom{n}{k} \left( f(k)-\sum^{k}_{j=0}(-1)^{n-j}\binom{k}{j}f(j)\right) B_{n-k}=0,\nonumber\\
&&\sum^{n}_{k=0}\binom{n}{k} \left( (-1)^{n-k}f(k)+\sum^{k}_{j=0}(-1)^{k-j}\binom{k}{j}f(j)\right) \left( E_{n-k}\left( \frac{1}{2}\right)-\frac{1}{2^{n-k}}\right)=0,
\en
where the first identity is given in \cite{SunZH}. 
\end{theorem}

\section{Applications of dual sequences to Bernoulli and Euler polynomials}

Let $\alpha\in{\bR}$. Denote 

\be\label{2-0-6}
C_{n,\alpha}(x)=(-1)^n\sum^{n}_{k=0}\binom{n}{k}a_{k}\left(x-\alpha\right)^{n-k}\quad \text{and}\quad C^{\ast}_{n,\alpha}(x)=(-1)^n\sum^{n}_{k=0}\binom{n}{k}a^{\ast}_{k}\left(x-\alpha\right)^{n-k}.
\ee
Then we have the following identities about $C_{n,\alpha}(x)$ and $C^{\ast}_{n,\alpha}(x)$, which is an extension of the main results shown in Sun \cite{Sun03}.

\begin{theorem}\label{thm:2-1}
Let $C_{n,\alpha}(x)$ and $C^{\ast}_{n,\alpha}(x)$ be defined as \eqref{2-0-6}, and let $k,\ell\in {\bN}\cup\{0\}$ and $x+y+z=1+2\alpha$. Then there hold

\bn\label{2-0-7}
&&\sum^{k}_{j=0}(-1)^{j}x^{k-j}\binom{k}{j}\frac{C_{\ell+j+1,\alpha}(y)}{\ell+j+1}+\sum^{\ell}_{j=0}(-1)^{j}x^{\ell-j}\binom{\ell}{j}\frac{C^{\ast}_{k+j+1,\alpha}(z)}{k+j+1}\nonumber\\
&=&\frac{a_{0}x^{k+\ell +1}}{(k+\ell+1)\binom{k+\ell}{k}}.
\en
In addition, we have 

\be\label{2-0-8}
\sum^{k}_{j=0}(-1)^{j}x^{k-j}\binom{k}{j}C_{\ell+j,\alpha}(y)=\sum^{\ell}_{j=0}(-1)^{j}x^{\ell-j}\binom{\ell}{j}C^{\ast}_{k+j,\alpha}(z)
\ee
and 

\bn\label{2-0-9}
&&\sum^{k}_{j=0}(-1)^{j}(\ell +j+1)x^{k-j+1}\binom{k+1}{j}C_{\ell+j,\alpha}(y)+\sum^{\ell}_{j=0}(-1)^{j}(k+j+1)x^{\ell-j+1}\binom{\ell+1}{j}C^{\ast}_{k+j,\alpha}(z)\nonumber\\
&=&(k+\ell+2)\left( (-1)^{k}C_{k+\ell+1, \alpha}(y)+(-1)^{\ell}C^{\ast}_{k+\ell+1,\alpha}(z)\right).
\en
\end{theorem}

\begin{proof}
Substituting \eqref{2-0-6} into the left-hand side of \eqref{2-0-7} yields

\bns
&&\sum^{k}_{j=0}\binom{k}{j}\frac{(-1)^{j}x^{k-j}}{\ell+j+1}\left( a_{0}(-1)^{\ell+j+1}\left(y-\alpha\right)^{\ell+j+1}+\sum^{\ell+j}_{i=0}a_{i+1}\binom{\ell+j+1}{i+1}(-1)^{\ell+j+1}\left(y-\alpha\right)^{\ell+j-i}\right)\\
&&+\sum^{\ell}_{j=0}\binom{\ell}{j}\frac{(-1)^{j}x^{\ell -j}}{k+j+1}\sum^{k+j+1}_{r=0}(-1)^{k+j+1}\binom{k+j+1}{r}\sum^{r}_{i=0}a_{i}\binom{r}{i}(-1)^{r}\left(z-\alpha\right)^{k+j+1-r}\\
&=&ca_{0}+\sum^{k+\ell}_{i=0}a_{i+1}\sum^{k}_{j=0}\binom{k}{j}\binom{\ell+j+1}{i+1}\frac{(-1)^{j}x^{k-j}}{\ell+j+1}(-1)^{\ell+j+1}\left(y-\alpha\right)^{\ell+j+1-(i+1)}\\
&&+\sum^{k+\ell+1}_{i=1}a_{i}\sum^{\ell}_{j=0}\sum^{k+j+1}_{r=i}\binom{\ell}{j}\binom{k+j+1}{r}\binom{r}{i}\frac{(-1)^{j}x^{\ell-j}}{k+j+1}(-\left(z-\alpha\right))^{k+j+1-r}\\
&=&ca_{0}+\sum^{k+\ell}_{i=0}\frac{c_{i}}{i+1}a_{i+1},
\ens
where

\bns
&&c=\sum^{k}_{j=0}\binom{k}{j}\frac{(-1)^{j}x^{k-j}}{\ell+j+1}(-1)^{\ell+j+1}\left(y-\alpha\right)^{\ell+j+1}\\
&&+\sum^{\ell}_{j=0}\sum^{k+j+1}_{r=0}\binom{\ell}{j}\binom{k+j+1}{r}\frac{(-1)^{j}x^{\ell-j}}{k+j+1}(-\left(z-\alpha\right))^{k+j+1-r},
\ens
which will be evaluated later, and 

\bns
&&(-1)^{\ell}c_{i}=\sum^{k}_{j=0}\binom{k}{j}\binom{\ell+j+1}{i+1}(i+1)\frac{x^{k-j}}{\ell+j+1}
\left(y-\alpha\right)^{\ell+j-i}\\
&&+\sum^{\ell}_{j=0}\sum^{k+j+1}_{r=i+1}\binom{\ell}{j}\binom{k+j+1}{r}\binom{r}{i+1}(i+1)\frac{(-1)^{\ell-j+1}x^{\ell-j}}{k+j+1}(-\left(z-\alpha\right))^{k+j+1-r}\\
&=&\sum^{k}_{j=0}\binom{k}{j}\binom{\ell+j}{i}x^{k-j}\left(y-\alpha\right)^{\ell+j-i}\\
&&+\sum^{\ell}_{j=0}\binom{\ell}{j}\binom{k+j}{i}(-1)^{\ell-j+1}x^{\ell -j}\sum^{k+j+1}_{r=i+1}\binom{k+j-i}{r-i-1}(-\left(z-\alpha\right))^{k+j-i-(r-i-1)}\\
&=&x^{k+\ell-i}\left(\sum^{k}_{j=0}\binom{k}{j}\binom{\ell+j}{i}\left( \frac{y-\alpha}{x}\right)^{\ell+j-i}\right.\\
&&\left. -\sum^{\ell}_{j=0}\binom{\ell}{j}\binom{k+j}{i}(-1)^{\ell -j}\left( 1+\frac{y-\alpha}{x}\right)^{k+j-i}\right)\\
&=&0,
\ens
where the last step is due to the expression in the last parenthesis is zero following from Lemma 3.1 of \cite{Sun03}.

Finally, we evaluate $c$ as follows.

\bns
&&c=(-1)^{\ell+1}\sum^{k}_{j=0}\binom{k}{j}x^{k-j}\frac{\left(y-\alpha\right)^{\ell +j+1}}{\ell +j+1}\\
&&+\sum^{\ell}_{j=0}\binom{\ell}{j}\frac{(-1)^{j}x^{\ell-j}}{k+j+1}\sum^{k+j+1}_{r=0}\binom{k+j+1}{r}
\left(-\left(z-\alpha\right)\right)^{k+j+1-r}\\
&=&(-1)^{\ell +1}\sum^{k}_{j=0}\binom{k}{j}x^{k-j}\int^{y-\alpha}_{0}t^{\ell +j}dt+
\sum^{\ell}_{j=0}\binom{\ell}{j}(-1)^{j}x^{\ell-j}\int^{x+y-\alpha}_{0}t^{k+j}dt\\
&=&(-1)^{\ell +1}\int^{y-\alpha}_{0}t^{\ell}(t+x)^{k}dt +(-1)^{\ell}\int^{x+y-\alpha}_{0}t^{k}(t-x)^{\ell}dt\\
&=&(-1)^{\ell }\left(-\int^{x+y-\alpha}_{x}(s-x)^{\ell}s^{k}ds+\int^{x+y-\alpha}_{0}(t-x)^{\ell}t^{k}dt\right)\\
&=&(-1)^{\ell }\int^{x}_{0}(t-x)^{\ell}t^{k}dt=(-1)^{\ell }\int^{1}_{0}(tx-x)^{\ell}(tx)^{k}xdt\\
&=&x^{k+\ell+1}\int^{1}_{0}(1-t)^{\ell}t^{k}dt=x^{k+\ell +1}\frac{\Gamma (k+1)\Gamma (\ell+1)}{\Gamma (k+\ell +2)}\\
&=&\frac{x^{k+\ell +1}}{(k+\ell +1)\binom{k+\ell}{k}}.
\ens

By taking partial derivative with respect to $y$ on the both sides of \eqref{2-0-7} and noting $z=1+2\alpha-x-y$, one may obtain \eqref{2-0-8}. Replacing $k$ and $\ell$ by $k+1$ and $\ell+1$ in \eqref{2-0-8} taking the partial derivative with respect to $y$, one may obtain \eqref{2-0-9}. The proof is complete. 
\end{proof}
\eop

The following corollary of Theorem \ref{thm:2-1} is equivalent to the results shown in Theorems 1.1 and 1.2 of \cite{Sun03}.

\begin{corollary}\label{cor:2-2} (\cite{Sun03}) 
Let $C_{n,0}(x)$ and $C^{\ast}_{n,0}(x)$ be defined as \eqref{2-0-6}, and let $k,\ell\in {\bN}\cup\{0\}$ and $x+y+z=1$. Then there holds

\be\label{2-0-2}
\sum^{k}_{j=0}(-1)^{j}x^{k-j}\binom{k}{j}\frac{C_{\ell+j+1,0}(y)}{\ell+j+1}+\sum^{\ell}_{j=0}(-1)^{j}x^{\ell-j}\binom{\ell}{j}\frac{C^{\ast}_{k+j+1,0}(z)}{k+j+1}
=\frac{a_{0}x^{k+\ell +1}}{(k+\ell+1)\binom{k+\ell}{k}}.
\ee
In addition, we have 

\be\label{2-0-3}
\sum^{k}_{j=0}(-1)^{j}x^{k-j}\binom{k}{j}C_{\ell+j,0}(y)=\sum^{\ell}_{j=0}(-1)^{j}x^{\ell-j}\binom{\ell}{j}C^{\ast}_{k+j,0}(z)
\ee
and 
\bns
&& \sum^{k}_{j=0}(-1)^{j}(\ell +j+1)x^{k-j+1}\binom{k+1}{j}C_{\ell+j,0}(y)+\sum^{\ell}_{j=0}(-1)^{j}(k+j+1)x^{\ell-j+1}\binom{\ell+1}{j}C^{\ast}_{k+j,0}(z)\nonumber\\
&=&(k+\ell+2)\left( (-1)^{k}C_{k+\ell+1, 0}(y)+(-1)^{\ell}C^{\ast}_{k+\ell+1,0}(z)\right).
\ens
If $a_{k}=B_{k}$ or $(-1)^{k}B_{k}$, then we have 

\bn\label{2-0-2-2}
&&(-1)^{\ell +1}\sum^{k}_{j=0}x^{k-j}\binom{k}{j}\frac{B_{\ell+j+1}(y)}{\ell+j+1}+(-1)^{k+1}\sum^{\ell}_{j=0}x^{\ell-j}\binom{\ell}{j}\frac{B_{k+j+1}(z)}{k+j+1}\nonumber\\
&=&\frac{x^{k+\ell +1}}{(k+\ell+1)\binom{k+\ell}{k}}
\en
In addition, we have 

\be\label{2-0-3-2}
(-1)^{\ell}\sum^{k}_{j=0}x^{k-j}\binom{k}{j}B_{\ell+j}(y)=(-1)^{k}\sum^{\ell}_{j=0}x^{\ell-j}\binom{\ell}{j}B_{k+j}(z)
\ee
and 

\bns
&&(-1)^{\ell}\sum^{k}_{j=0})\ell +j+1)x^{k-j}\binom{k+1}{j}B_{\ell+j}(y)+(-1)^{k}\sum^{\ell}_{j=0}(k+j+1)x^{\ell-j}\binom{\ell+1}{j}B_{k+j}(z)\\
&=&(k+\ell+2)\left( (-1)^{k}B_{k+\ell+1}(y)+(-1)^{\ell}B_{k+\ell+1}(z)\right).
\ens
\end{corollary}

The conjugate Bernoulli polynomials $\tilde B_{n}(x)$ is introduced in \cite{HS15} by their generating function as follows.

\be\label{6-5}
\frac{e^{xt}}{1+\frac{1+t-e^{t}}{t}}=\sum^\infty_{n=0}\tilde B_n(x)\frac{t^n}{n!},
\ee
where the first few terms of the conjugate polynomial sequence $\{ \tilde B_{n}\}_{n\geq 0}$ are 

\bns
&&\tilde B_{0}(x)=1,\\
&&\tilde B_{1}(x)=x+\frac{1}{2},\\
&&\tilde B_{2}(x)=x^{2}+x+\frac{5}{6},\\
&&\tilde B_{3}(x)=x^{3}+\frac{3}{2}x^{2}+\frac{5}{2}x+2,\\
&&\tilde B_{4}(x)=x^{4}+2x^{3}+5x^{2}+8x+\frac{191}{30},\\
&&\tilde B_{5}(x)=x^{5}+\frac{5}{2}x^{4}+\frac{25}{3}x^{3}+20x^{2}+\frac{191}{6}x+\frac{76}{3}, etc.
\ens

Let $\{ \tilde B_{n}(x)\}_{n\geq 0}$ be the conjugate Bernoulli polynomial sequence defined by \eqref{6-5}, and let the conjugate Bernoulli numbers $\{ \tilde B_{n}\}_{n\geq 0}$ be defined by $\tilde B_{n}=\tilde B_{n}(0)$. Then we may find that   

\be\label{6-6}
\tilde B_{n}(x)=\sum^{n}_{k=0}\binom{n}{k}x^{n-k}\tilde B_{k}
\ee
for all $n\geq 0$.

We define dual sequence of the conjugate Bernoulli number sequence with respect to inverse matrix $R_{1}$. Then the corresponding dual polynomial sequence is  

\be\label{6-7}
\tilde B^{\ast}_{n}(x):=\sum^{n}_{k=0}\binom{n}{k}x^{n-k}\tilde B^{\ast}_{k}
\ee
for all $n\geq 0$. From Theorem \ref{thm:2-1}, we obtain 

\begin{corollary}\label{cor:2-3}
Let $k,\ell \in{\bN}\cup \{0\}$ and $x+y+z=1$. Then there holds 

\bn\label{2-12}
&&(-1)^{k}\sum^{k}_{j=0}\binom{k}{j}x^{k-j}(-1)^{j+1}\frac{\tilde B_{\ell+j+1}(-y)}{\ell+j+1}+\sum^{\ell}_{j=0}\binom{\ell}{j}x^{\ell -j}\frac{\tilde B^{\ast}_{k+j+1}(-z)}{k+j+1}\nonumber\\
&=&\frac{(-1)^{k +1}x^{k+\ell +1}}{(k+\ell +1)\binom{k+\ell}{k}},
\en
where $\tilde B_{n}(t)$ are defined by \eqref{6-5}, and $\tilde B^{\ast}_{n}(x)$ are defined by \eqref{6-7}. Also 

\be\label{2-13}
(-1)^{k}\sum^{k}_{j=0}\binom{k}{j}(-x)^{k-j}\tilde B_{\ell+j}(-y)=(-1)^{\ell}\sum^{\ell}_{j=0}\binom{\ell}{j}(-x)^{\ell-j}\tilde B^{\ast}_{k+j}(-z)
\ee
and

\bn\label{2-14}
&&(-1)^{\ell}\sum^{k}_{j=0}\binom{k+1}{j}(-x)^{k-j+1}(\ell+j+1)\tilde B_{\ell+j}(-y)\nonumber\\
&&+(-1)^{k}\sum^{\ell}_{j=0}\binom{\ell+1}{j}(-x)^{\ell-j+1}(k+j+1)\tilde B^{\ast}_{k+j}(-z)\nonumber\\
&=& (k+\ell+2)((-1)^{\ell+1}\tilde B_{k+\ell+1}(-y)+(-1)^{k+1}\tilde B^{\ast}_{k+\ell_+1}(-z)).
\en
\end{corollary}

In \eqref{2-0-6}, substituting $\alpha =1/2$ and 

\be\label{2-15}
a_{k}=E_{k}\left( \frac{1}{2}\right)-\frac{1}{2^{k}}\quad \text{and} \quad a^{\ast}_{k}=(-1)^{k}\left( E_{k}\left( \frac{1}{2}\right) +\frac{3^{k}-2}{2^{k}}\right),
\ee
where the duals $a^{\ast}_{k}$ are derived by using \eqref{0-16} of Theorem \ref{thm:1.2}, we obtain 

\[
A_{n}(x)=(-1)^{n}E_{n}(x)-(-1)^{n}x^{n}\quad \text{and}\quad A^{\ast}_{n}(x)=E_{n}(-x+1)+(2-x)^{n}-2(1-x)^{n}.
\]
Hence, Theorem \ref{thm:2-1} implies 

\begin{corollary}\label{cor:2-4}
Let $C_{n,\alpha}(x)$ and $C^{\ast}_{n,\alpha}(x)$ be defined as \eqref{2-0-6} with $\alpha =1/2$, and let $k,\ell\in {\bN}\cup\{0\}$ and $x+y+z=2$. For $a_{k}$ and $a^{\ast}_{k}$ 
shown in \eqref{2-15}, there hold

\bn\label{2-16}
&&\sum^{k}_{j=0}(-1)^{\ell +1}x^{k-j}\binom{k}{j}\frac{E_{\ell+j+1}(y)}{\ell+j+1}+\sum^{\ell}_{j=0}(-1)^{j}x^{\ell-j}\binom{\ell}{j}\frac{E_{k+j+1}(-z+1)}{k+j+1}\nonumber\\
&=&(-1)^{\ell+1}\int^{y}_{0}t^{\ell}(t+x)^{k}dt-\int^{x+y}_{0}t^{k}(x-t)^{\ell}dt+2\int^{x+y-1}_{0}t^{k}(x-t)^{\ell}dt.\nonumber\\
\en
In addition, we have 

\bn\label{2-17}
&&(-1)^{\ell}\sum^{k}_{j=0}x^{k-j}\binom{k}{j}E_{\ell+j}(y)-(-1)^{\ell}y^{\ell}(x+y)^{k}\nonumber\\
&=&\sum^{\ell}_{j=0}(-1)^{j}x^{\ell-j}\binom{\ell}{j}E_{k+j}(1-z)+(x+y)^{k}(-y)^{\ell}-2(x+y-1)^{k}(1-y)^{\ell}.
\en
\end{corollary}

The above identities of polynomial sequences can be used to establish identities of number sequences. For instance, if $x=1$, $y=0$, and $z=1$, then \eqref{2-16} yields the number sequence identity

\bns
&&\sum^{k}_{j=0}(-1)^{\ell +1}x^{k-j}\binom{k}{j}\frac{E_{\ell+j+1}(0)}{\ell+j+1}\\
&&+\sum^{\ell}_{j=0}(-1)^{j}x^{\ell-j}\binom{\ell}{j}\frac{E_{k+j+1}(0)}{k+j+1}=-\frac{1}{(k+\ell+1)\binom{k+\ell}{k}}.
\ens
If $x=1$ and $y=z=1/2$, from \eqref{2-16} there holds 

\bns
&&\sum^{k}_{j=0}(-1)^{\ell +1}x^{k-j}\binom{k}{j}\frac{E_{\ell+j+1}(1/2)}{\ell+j+1}+\sum^{\ell}_{j=0}(-1)^{j}x^{\ell-j}\binom{\ell}{j}\frac{E_{k+j+1}(1/2)}{k+j+1}\\
&=&-\frac{1}{(k+\ell+1)\binom{k+\ell}{k}}+2B\left( \frac{1}{2}, k+1, \ell+1\right)-2B\left( \frac{3}{2}, k+1, \ell+1\right),
\ens
where $B(\alpha, a,b)=\int^{\alpha}_{0}t^{a-1}(1-t)^{b-1}dt$ is an incomplete beta function. From the last two identities, we have 

\bns
&&\sum^{k}_{j=0}(-1)^{\ell +1}x^{k-j}\binom{k}{j}\frac{E_{\ell+j+1}(1/2)-E_{\ell+j+1}(0)}{\ell+j+1}+\sum^{\ell}_{j=0}(-1)^{j}x^{\ell-j}\binom{\ell}{j}\frac{E_{k+j+1}(1/2)-E_{k+j+1}(0)}{k+j+1}\\
&&=2\left(B\left( \frac{1}{2}, k+1, \ell+1\right)-B\left( \frac{3}{2}, k+1, \ell+1\right)\right).
\ens

Let $A$ and $B$ be any $m\times m$ and $n\times n$ square matrices, respectively. Cheon and Kim \cite{CK} use the notation $\oplus$ for the direct sum of matrices $A$ and $B$:

\be\label{-1}
A\oplus B=\left[ \begin{array} {ll} A & 0\\ 0 &B\end{array}\right].
\ee
Hence, we may obtain a matrix form of Lemma 3.2 of Pan and Sun \cite{PS}

 \begin{theorem}\label{thm:6-5} 
Let $n$ be any positive integer, and let $\tilde B(t)$ and $P[t]$ be defined as 

\bns
&&\tilde B(t):=(\tilde B_{0}(t), \tilde B_{1}(t), \cdots)\quad and \\ 
&&P[t]:=\left( \binom{n}{k} t^{n-k} \right)_{n,k\geq 0},
\ens
respectively. Denote $(x)=(1,x,x^{2},\ldots)^{T}$, $D=diag(0,1,1/2, \ldots)$, 
and $[X]=[0]\oplus \left[ \frac{x^{n-k}}{k}\right]_{1\leq k\leq n}$, i.e., 

\[
[X]=\left[\begin{array}{lllll}0 & 0 &0 & \cdots & 0 \\ 0 & 1 & 0 & \cdots & 0 \\ 0 & x & \frac{1} 2 & \cdots & 0\\ \vdots & \vdots & \vdots & \ddots & \vdots \\ 0 & x ^{n-1} & \frac{x^{n-2}}2 & \cdots & \frac{1}n\\ \vdots & \vdots & \vdots & \ddots & \vdots \end{array}\right].
\]
Then

\be\label{-2}
[X] \tilde B(x+y)=P[x]D\tilde B(y)+[X](x),
\ee
which implies 

\be\label{-3}
\sum_{k=1}^{n} \frac{ \tilde B_{k} (x+y)}{k} x^{n-k} = \sum_{l=1}^{n} {n \choose l} \frac{\tilde B_{l}(y)}{l} x^{n-l}+ H_{n} x^{n}
\ee
for all $n\geq 0$ shown in  \cite{PS}.
\end{theorem}
\begin{proof} Noting that $\tilde B_{0}(y) = 1$, the left-hand side of \eqref{-2} can be written as 

\bns
&&LHS\,\, of \,\, \eqref{-2}\\
&&=[X]P[x]\tilde B(y)\\
&&= [X]\left[\begin{array}{l}1 \\x + {1 \choose 1} \tilde B_{1}(y)  \\ x^{2} +  {2 \choose 1}x \tilde B_  {1}(y) + {2\choose 2}\tilde B_  {2}(y) \\ \vdots  \\ x^{n} + {n \choose 1}x^{n-1} \tilde B_  {1}(y) + \cdots + {n \choose n} \tilde B_  {n}(y)\\\vdots\end{array}\right]\\
&&=  [X]\left( \left[\begin{array}{l}1\\ x \\ x^{2} \\ \vdots \\ x^{n}\\ \vdots\end{array}\right]+\left[\begin{array}{l} 0 \\{1 \choose 1} \tilde B_  {1}(y)  \\ {2 \choose 1}x \tilde B_  {1}(y) + {2\choose 2} \tilde B_  {2}(y) \\ \vdots  \\ {n \choose 1}x^{n-1} \tilde B_  {1}(y) + \cdots + {n \choose n} \tilde B_  {n}(y)\\ \vdots \end{array}\right]\right)\\
& &= \left[\begin{array}{l} 0 \\  x  \\ ( 1+ \frac{1}2) x^{2} \\ \vdots \\ (1+ \frac{1}2 + \cdots + \frac{1}n) x ^{n}  \\ \vdots \end{array}\right] +[X] 
\left[\begin{array}{l} 0 \\{1 \choose 1} \tilde B_  {1}(y)  \\ {2 \choose 1}x \tilde B_  {1}(y) + {2\choose 2} \tilde B_  {2}(y) \\ \vdots  \\ {n \choose 1}x^{n-1} \tilde B_  {1}(y) + \cdots + {n \choose n} \tilde B_  {n}(y)\\ \vdots \end{array}\right],
\ens
where the second term can be written as 
\bns
&&\left[\begin{array}{l}0 \\
\left( \binom{1}{1}+\frac{1} 2 {2 \choose 1}\right) \tilde B_  {1}(y)x +  \frac{1} 2 \binom{2}{2}\tilde B_  {2}(y) \\ \vdots  \\ 
\left( {1 \choose 1} +\frac{1}{2}\binom{2}{1}+\cdots +\frac{1}{n}\binom{n}{1}\right)\tilde B_  {1}(y)x^{n-1} +  \left(\frac{1} 2 {2 \choose 2}+\frac{1}{3}\binom{3}{2}
+\cdots + \frac{1}{n}\binom{n}{2}\right) \tilde B_  {2}(y)x^{n-2}+\cdots +  \frac{1}n  {n \choose n} \tilde B_  {n}(y)\\  \vdots \end{array}\right]\\
&&=\left[\begin{array}{l}0 \\
\binom{2}{1}\tilde B_  {1}(y)x+\frac{1}{2}\binom{2}{2}\tilde B_  {2}(y)\\ \vdots \\
\binom{n}{1}\tilde B_  {1}(y)x^{n-1}+\frac{1}{2}\binom{n}{2}\tilde B_  {2}(y)x^{n-2}+\cdots \frac{1}{n}\binom{n}{n}\tilde B_  {n}(y)\\ \vdots \end{array}\right].
\ens
In the last step we use the identities $\frac{1}{n}\binom{n}{k}=\frac{1}{k}\binom{n-1}{k-1}$ and  

\bns
&&\binom{n}{k}=\binom{n-1}{k-1}+\binom{n-1}{k}=\binom{n-1}{k-1}+\binom{n-2}{k-1}+\binom{n-3}{k}=\cdots\\
&=&\binom{n-1}{k-1}+\binom{n-2}{k-1}+\cdots +\binom{k}{k-1}+\binom{k}{k}
\ens
for $ k, n\in {\bN}$ with $1\leq k\leq n$. Since the last matrix can be written as $P[x]DB(y)$, we obtain \eqref{-2}. By multiplying the matrices of \eqref{-2}, 

$$ 
LHS \,\, of\,\, \eqref{-3}=  H_{n} x^{n}+ \sum_{l=1}^{n} {n \choose l} \frac{\tilde B_  {l}(y)}{l} x^{n-l} = \sum_{l=1}^{n} {n \choose l} \frac{\tilde B_  {l}(y)}{l} x^{n-l}+ H_{n} x^{n}, 
$$
which is the RHS of \eqref{-3}. 
\end{proof}
\eop

\end{document}